\documentclass[12pt]{article}
\usepackage{amsmath,amssymb,amsthm}
\numberwithin{equation}{section}
\usepackage{hyperref}
\usepackage{units}
\usepackage{color}
\usepackage[T1]{fontenc}
\usepackage[utf8]{inputenc}
\usepackage{authblk}
\usepackage{bm}
\usepackage{graphicx}

\usepackage{datetime}

\usepackage[toc,page]{appendix}

\newcommand{\Z}{{\mathbb Z}}

\newtheorem{thm}{Theorem}

\newtheorem{lemma}{Lemma}

\newtheorem{cor}[thm]{Corollary}

\title{Weighted Tribonacci sums}

\author[]{Kunle Adegoke \thanks{adegoke00@gmail.com}}
\affil{Department of Physics and Engineering Physics, \mbox{Obafemi Awolowo University}, 220005 Ile-Ife, Nigeria}

\begin{document}

\date{}

\maketitle

\begin{abstract}
\noindent We derive various weighted summation identities, including binomial and double binomial identities, for Tribonacci numbers. Our results contain some previously known results as special cases.
\end{abstract}

%\tableofcontents
% \listoffigures
%\chead{\thepage}
\section{Introduction}
For $m\ge 3$, the Tribonacci numbers are defined by
\begin{equation}\label{eq.ffnh4ol}
T_m=T_{m-1}+T_{m-2}+T_{m-3}\,,\quad T_0=0,\,T_1=T_2=1\,.
\end{equation}
By writing $T_{m-1}=T_{m-2}+T_{m-3}+T_{m-4}$ and eliminating $T_{m-2}$ and $T_{m-3}$ between this recurrence relation and the recurrence relation~\eqref{eq.ffnh4ol}, a useful alternative recurrence relation is obtained for $m\ge 4$: 
\begin{equation}\label{eq.o6ns6qw}
T_m=2T_{m-1}-T_{m-4}\,,\quad T_0=0\,,\quad T_1=T_2=1\,,\quad T_3=2\,.
\end{equation}
Extension of the definition of $T_m$ to negative subscripts is provided by writing the recurrence relation~\eqref{eq.o6ns6qw} as
\begin{equation}
T_{-m}=2T_{-m+3}-T_{-m+4}\,.
\end{equation}
Anantakitpaisal and Kuhapatanakul~\cite{ananta16} proved that
\begin{equation}
T_{-m}=T_{m-1}{}^2-T_{m-2}T_m.
\end{equation}
The following identity (Feng~\cite{feng11}, equation~(3.3); Shah~\cite{shah11}, (ii)) is readily established by the principle of mathematical induction:
\begin{equation}\label{eq.c1sfqfe}
T_{m + r}  = T_{r} T_{m-2}  + (T_{r - 1}  + T_{r} )T_{m - 1}  + T_{r+1} T_{m}\,.
\end{equation}
Irmak and Alp~\cite{irmak13} derived the following identity for Tribonacci numbers with indices in arithmetic progression:
\begin{equation}\label{eq.yssh328}
T_{tm + r}  = \lambda_1(t) T_{t(m - 1) + r}  + \lambda_2(t) T_{t(m - 2) + r}  + \lambda_3(t) T_{t(m - 3) + r}\,,
\end{equation}
where,
\[
\lambda_1(t)= \alpha ^t  + \beta ^t  + \gamma ^t ,\quad \lambda_2(t)= -(\alpha \beta )^t  - (\alpha \gamma )^t  - (\beta \gamma )^t ,\quad \lambda_3(t)= (\alpha \beta \gamma )^t\,,
\]
where $\alpha$, $\beta$ and $\gamma$ are the roots of the characteristic polynomial of the Tribonacci sequence \mbox{$x^3-x^2-x-1$}. Thus,
\[
\alpha  = \frac{1}{3}\left( {1 + \sqrt [3]{19 + 3\sqrt {33} }  + \sqrt [3]{19 - 3\sqrt {33} } } \right)\,,
\]
\[
\beta  = \frac{1}{3}\left( {1 + \omega \sqrt [3]{19 + 3\sqrt {33} }  + \omega ^2 \sqrt [3]{19 - 3\sqrt {33} } } \right)
\]
and
\[
\gamma  = \frac{1}{3}\left( {1 + \omega ^2 \sqrt[3] {19 + 3\sqrt {33} }  + \omega \sqrt [3]{19 - 3\sqrt {33} } } \right)\,,
\]
where $\omega=\exp{(2i\pi/3)}$ is a primitive cube root of unity. Note that $\lambda_1(t)$, $\lambda_2(t)$ and $\lambda_3(t)$ are integers for any positive integer $t$ \cite{irmak13}; in particular, $\lambda_1(1)=1=\lambda_2(1)=\lambda_3(1)$.

\section{Weighted sums}
\begin{lemma}[\cite{adegoke18}, Lemma 2]\label{lem.s9jfs7n}
Let $\{X_m\}$ be any arbitrary sequence, where $X_m$, $m\in\Z$, satisfies a second order recurrence relation $X_m=f_1X_{m-a}+f_2X_{m-b}$, where $f_1$ and $f_2$ are arbitrary non-vanishing complex functions, not dependent on $m$, and $a$ and $b$ are integers. Then,
\begin{equation}\label{eq.mf_1f_2b9zk}
f_2\sum_{j = 0}^k {\frac{{X_{m - ka  - b  + a j} }}{{f_1^j }}}  = \frac{{X_m }}{{f_1^k }} - f_1X_{m - (k + 1)a }\,,
\end{equation}
\begin{equation}\label{eq.cgldajj}
f_1\sum_{j = 0}^k {\frac{{X_{m - kb  - a  + b j} }}{{f_2^j }}}  = \frac{{X_m }}{{f_2^k }} - f_2X_{m - (k + 1)b } 
\end{equation}
and
\begin{equation}\label{eq.n2n4ec3}
\sum_{j = 0}^k { \frac{X_{m - (b - a)k + a + (b - a)j}}{(-f_2/f_1)^j} }  = \frac{f_1 X_m}{(-f_2/f_1)^k}  + f_2 X_{m - (k + 1)(b - a)}\,.
\end{equation}
for $k$ a non-negative integer.
\end{lemma}
\begin{thm}\label{thm.mt8vm2p}
The following identities hold for any integers $m$ and $k$:
\begin{equation}\label{eq.ppp3zo9}
\sum_{j = 0}^k {2^{ - j} T_{m - k - 4 + j} }  = 2T_{m - k - 1}  - 2^{ - k} T_m\,, 
\end{equation}
\begin{equation}\label{eq.q2makc7}
2\sum_{j = 0}^k {( - 1)^j T_{m - 4k - 1 + 4j} }  = ( - 1)^k T_m  + T_{m - 4k - 4} 
\end{equation}
and
\begin{equation}\label{eq.lgpm86c}
\sum_{j = 0}^k {2^j T_{m - 3k + 1 + 3j} }  = 2^{k + 1} T_m  - T_{m - 3k - 3}\,. 
\end{equation}

\end{thm}
\begin{proof}
From the recurrence relation~\eqref{eq.o6ns6qw}, make the identifications $f_1=2$, $f_2=-1$, $a=1$ and $b=4$ and use these in Lemma~\ref{lem.s9jfs7n} with $X=T$.
\end{proof}
Particular instances of identities~\eqref{eq.ppp3zo9}--\eqref{eq.lgpm86c} are the following identities:
\begin{equation}
\sum_{j = 0}^k {2^{ - j} T_j }  = 4 - 2^{ - k} T_{k + 4}\,,
\end{equation}
giving,
\begin{equation}
\sum_{j = 0}^\infty  {2^{ - j} T_j }  = 4\,,
\end{equation}
and
\begin{equation}
2\sum_{j = 0}^k {( - 1)^j T_{4j} }  = ( - 1)^k T_{4k + 1}  - 1
\end{equation}
and
\begin{equation}
\sum_{j = 0}^k {2^j T_{3j} }  = 2^{k + 1} T_{3k - 1}\,.
\end{equation}
\begin{lemma}[Partial sum of an $n^{th}$ order sequence]\label{lemma.qm8k37h}
Let $\{X_j\}$ be any arbitrary sequence, where $X_j$, $j\in\Z$, satisfies a $n^{th}$~order recurrence relation $X_j=f_1X_{j-c_1}+f_2X_{j-c_2}+\cdots+f_nX_{j-c_n}=\sum_{m=1}^n f_mX_{j-c_m}$, where $f_1$, $f_2$, $\ldots$, $f_n$ are arbitrary non-vanishing complex functions, not dependent on $j$, and $c_1$, $c_2$, $\ldots$, $c_n$ are fixed integers. Then, the following summation identity holds for arbitrary $x$ and non-negative integer $k$ :
\[
\sum_{j = 0}^k {x^j X_j }  = \frac{{\sum_{m = 1}^n {\left\{ {x^{c_m } f_m \left( {\sum_{j = 1}^{c_m } {x^{ - j} X_{ - j} }  - \sum_{j = k - c_m  + 1}^k {x^j X_j } } \right)} \right\}} }}{{1 - \sum_{m = 1}^n {x^{c_m } f_m } }}\,.
\]

\end{lemma}
\begin{proof}
Recurrence relation:
\[
X_j=\sum_{m=1}^n f_mX_{j-c_m}\,.
\]
We multiply both sides by $x^j$ and sum over $j$ to obtain
\[
\sum_{j = 0}^k {x^j X_j }  = \sum_{m = 1}^n {\left( {f_m \sum_{j = 0}^k {x^j X_{j - c_m } } } \right)}  = \sum_{m = 1}^n {\left( {x^{c_m } f_m \sum_{j =  - c_m }^{k - c_m } {x^j X_j } } \right)}\,, 
\]
after shifting the summation index $j$.
Splitting the inner sum, we can write
\[
\sum_{j = 0}^k {x^j X_j }  = \sum_{m = 1}^n {x^{c_m } f_m \left( {\sum_{j =  - c_m }^{ - 1} {x^j X_j }  + \sum_{j = 0}^k {x^j X_j }  + \sum_{j = k + 1}^{k - c_m } {x^j X_j } } \right)}\,.
\]
Since
\[
\sum_{j =  - c_m }^{ - 1} {x^j X_j }  \equiv \sum_{j = 1}^{c_m } {x^{ - j} X_{ - j} }\mbox{ and}\quad\sum_{j = k + 1}^{k - c_m } {x^j X_j }  \equiv  - \sum_{j = k - c_m  + 1}^k {x^j X_j }\,,
\]
the preceding identity can be written
\[
\sum_{j = 0}^k {x^j X_j }  = \sum_{m = 1}^n {x^{c_m } f_m \left( {\sum_{j = 1}^{c_m } {x^{ - j} X_{ - j} }  + \sum_{j = 0}^k {x^j X_j }  - \sum_{j = k - c_m  + 1}^k {x^j X_j } } \right)}\,.
\]
Thus, we have
\[
S = \sum_{m = 1}^n {x^{c_m } f_m \left( {\sum_{j = 1}^{c_m } {x^{ - j} X_{ - j} }  + S - \sum_{j = k - c_m  + 1}^k {x^j X_j } } \right)}\,, 
\]
where 
\[
S = S_k (x) = \sum_{j = 0}^k {x^j X_j }\,.
\]
Removing brackets, we have
\[
S = \sum_{m = 1}^n {x^{c_m } f_m \left( {\sum_{j = 1}^{c_m } {x^{ - j} X_{ - j} }  - \sum_{j = k - c_m  + 1}^k {x^j X_j } } \right)}  + S\sum_{m = 1}^n {x^{c_m } f_m }\,,
\]
from which the result follows by grouping the $S$ terms.
\end{proof}
\begin{lemma}[Generating function]\label{lemma.v1j9biq}
Under the conditions of Lemma~\ref{lemma.qm8k37h}, if additionally $x^kX_k$ vanishes in the limit as $k$ approaches infinity, then
\[
S_\infty  (x) = \sum_{j = 0}^\infty  {x^j X_j }  = \frac{{\sum_{m = 1}^n {\left( {x^{c_m } f_m \sum_{j = 1}^{c_m } {x^{ - j} X_{ - j} } } \right)} }}{{1 - \sum_{m = 1}^n {x^{c_m } f_m } }}\,,
\]
so that $S_\infty(x)$ is a generating function for the sequence $\{X_j\}$.
\end{lemma}
\begin{thm}[Sum of Tribonacci numbers with indices in arithmetic progression]\label{thm.k1j38a4}
For arbitrary $x$, any integers $t$ and $r$ and any non-negative integer $k$, the following identity holds:
\[
\begin{split}
\left(1 - \lambda_1(t)x - \lambda_2(t)x^2  - \lambda_3(t)x^3 \right)\sum_{j = 0}^k {x^j T_{tj + r} }  &= T_r  + (x\lambda_2(t)  + x^2 \lambda_3(t) )T_{r - t}\\
&\qquad + x\lambda_3(t) T_{r - 2t}- x^{k + 1} T_{(k + 1)t + r}\\ 
&\qquad - x^{k + 2} (\lambda_2(t)  + x\lambda_3(t) )T_{kt + r}\\ 
&\qquad- x^{k + 2} \lambda_3(t) T_{(k - 1)t + r}\,,
\end{split}
\]
where,
\[
\lambda_1(t)  = \alpha ^t  + \beta ^t  + \gamma ^t ,\quad \lambda_2(t)  = -(\alpha \beta )^t  - (\alpha \gamma )^t  - (\beta \gamma )^t ,\quad \lambda_3(t)  = (\alpha \beta \gamma )^t\,,
\]
where $\alpha$, $\beta$ and $\gamma$ are the roots of the characteristic polynomial of the Tribonacci sequence \mbox{$x^3-x^2-x-1$}.

\end{thm}
\begin{proof}
Write identity~\eqref{eq.yssh328} as $X_j  = f_1 X_{j - 1}  + f_2 X_{j - 2}  + f_3 X_{j - 3} $ and identify the sequence $\{X_j\}=\{T_{tj+r}\}$ and the constants $c_1=1$, $c_2=2$, $c_3=3$ and the functions $f_1=\lambda_1(t)$, $f_2=\lambda_2(t)$, $f_3=\lambda_3(t)$, and use these in Lemma~\ref{lemma.qm8k37h}.

\end{proof}
\begin{cor}[Generating function of the Tribonacci numbers with indices in arithmetic progression]
For any integers $t$ and $r$, any non-negative integer $k$ and arbitrary $x$ for which $x^kT_k$ vanishes as $k$ approaches infinity , the following identity holds:
\[
\sum_{j = 0}^\infty  {x^j T_{tj + r} }  = \frac{{T_r  + (x\lambda_2  + x^2 \lambda_3 )T_{r - t}  + x\lambda_3 T_{r - 2t} }}{{1 - \lambda_1x - \lambda_2x^2  - \lambda_3x^3 }}\,,
\]
where,
\[
\lambda_1  = \alpha ^t  + \beta ^t  + \gamma ^t ,\quad \lambda_2  = -(\alpha \beta )^t  - (\alpha \gamma )^t  - (\beta \gamma )^t ,\quad \lambda_3  = (\alpha \beta \gamma )^t\,,
\]
where $\alpha$, $\beta$ and $\gamma$ are the roots of the characteristic polynomial of the Tribonacci sequence \mbox{$x^3-x^2-x-1$}.
\end{cor}
Many instances of Theorem~\ref{thm.k1j38a4} may be explored. In particular, we have
\begin{equation}\label{eq.wpt4d3q}
\begin{split}
(\lambda_1(t)+\lambda_2(t)+\lambda_3(t)-1)\sum_{j = 0}^k {T_{tj + r} } &=  - T_r  - (\lambda_2(t)  + \lambda_3(t) )T_{r - t}\\
&\quad  - \lambda_3(t) T_{r - 2t}  + T_{(k + 1)t + r}\\
&\qquad+ (\lambda_2(t)  + \lambda_3(t) )T_{kt + r}  + \lambda_3(t) T_{(k - 1)t + r}\,,
\end{split}
\end{equation}
which at $r=0$ gives
\begin{equation}\label{eq.ai4p8m1}
\begin{split}
(\lambda_1(t)+\lambda_2(t)+\lambda_3(t)-1)\sum_{j = 0}^k {T_{tj} }  &=  - (\lambda_2(t)  + \lambda_3(t) )(T_{t - 1} ^2  - T_{t - 2} T_t )\\
&\quad - \lambda_3(t) (T_{2t - 1} ^2  - T_{2t - 2} T_{2t} )+ T_{(k + 1)t}\\
&\qquad+ (\lambda_2(t)  + \lambda_3(t) )T_{kt}  + \lambda_3(t) T_{(k - 1)t}\;;
\end{split}
\end{equation}
and
\begin{equation}\label{eq.qswmqj2}
\begin{split}
(1+\lambda_1(t)-\lambda_2(t)+\lambda_3(t))\sum_{j = 0}^k {(-1)^jT_{tj + r} }  &= T_r  + (\lambda_3(t)  - \lambda_2(t) )T_{r - t}\\
&\quad  - \lambda_3(t) T_{r - 2t}  + ( - 1)^k T_{(k + 1)t + r}\\ 
&\qquad + ( - 1)^k (\lambda_3(t) - \lambda_2(t) )T_{kt + r}\\
&\qquad  - ( - 1)^k \lambda_3(t) T_{(k - 1)t + r}\,,
\end{split}
\end{equation}
which at $r=0$ gives
\begin{equation}
\begin{split}
(1+\lambda_1(t)-\lambda_2(t)+\lambda_3(t))\sum_{j = 0}^k {(-1)^jT_{tj} } &= (\lambda_3(t)  - \lambda_2(t) )(T_{t - 1} ^2  - T_{t - 2} T_t )\\
&\qquad - \lambda_3(t) (T_{2t - 1} ^2  - T_{2t - 2} T_{2t} )+ ( - 1)^k T_{(k + 1)t}\\
&\qquad + ( - 1)^k (\lambda_3(t)  - \lambda_2(t) )T_{kt}\\
&\qquad  - ( - 1)^k \lambda_3(t) T_{(k - 1)t}\,. 
\end{split}
\end{equation}
Many previously known results are particular instances of the identities~\eqref{eq.wpt4d3q} and~\eqref{eq.qswmqj2}. For example, Theorem~5 of~\cite{kilic08} is obtained from identity~\eqref{eq.ai4p8m1} by setting $t=4$. Sums of Tribonacci numbers with indices in arithmetic progression are also discussed in references~\cite{frontczak18, irmak13, kilic08} and references therein, using various techniques.

\bigskip

Weighted sums of the form $\sum_{j=0}^k{j^pT_{tj+r}}$, where $p$ is a non-negative integer, may be evaluated by setting $x=e^y$ in the identity of~Theorem~\ref{thm.k1j38a4}, differentiating both sides $p$ times with respect to $y$ and then setting $y=0$. The simplest examples in this category are the following:
\begin{equation}
\begin{split}
2\sum_{j = 0}^k {jT_{j + r} }  &= - T_{r - 2} + 3T_{r + 1}  + (k - 1)T_{k + r - 1}\\
&\qquad  + (2k - 1)T_{k + r}  + (k - 2)T_{k + r + 1}
\end{split}
\end{equation}
and
\begin{equation}
\begin{split}
2\sum_{j = 0}^k {j^2 T_{j + r} }  &=- 3T_{r - 1}  - 5T_r - 6T_{r + 1}\\
&\quad + (k^2  - 2k + 3)T_{k + r - 1}+(2k^2  - 2k + 5)T_{k + r}\\ 
&\qquad + (k^2  - 4k + 6)T_{k + r + 1}\,,
\end{split}
\end{equation}
with the particular cases
\begin{equation}
2\sum_{j = 0}^k {jT_j }  = 2 + (k - 1)T_{k - 1}  + (2k - 1)T_k  + (k - 2)T_{k + 1}
\end{equation}
and
\begin{equation}
\begin{split}
2\sum_{j = 0}^k {j^2 T_{j} }  &= - 6 + (k^2  - 2k + 3)T_{k + r - 1}\\
&\quad +(2k^2  - 2k + 5)T_{k}\\ 
&\qquad + (k^2  - 4k + 6)T_{k + 1}\,.
\end{split}
\end{equation}

\section{Weighted binomial sums}
\begin{lemma}[\cite{adegoke18}, Lemma 3]\label{lem.i84yg3s}
Let $\{X_m\}$ be any arbitrary sequence. Let $X_m$, $m\in\Z$, satisfy a second order recurrence relation $X_m=f_1X_{m-a}+f_2X_{m-b}$, where $f_1$ and $f_2$ are non-vanishing complex functions, not dependent on $m$, and $a$ and $b$ are integers. Then,
\begin{equation}
\sum_{j = 0}^k {\binom kj\left( {\frac{f_1}{f_2}} \right)^j X_{m - bk  + (b  - a )j} }  = \frac{{X_m }}{{f_2^k }}\,,
\end{equation}
\begin{equation}
\sum_{j = 0}^k {\binom kj\frac{{X_{m + (a - b)k + bj} }}{{( - f_2 )^j }}}  = \left( { - \frac{{f_1 }}{{f_2 }}} \right)^k X_m
\end{equation}
and
\begin{equation}\label{eq.fnwrzi3}
\sum_{j = 0}^k {\binom kj\frac{{X_{m + (b - a)k + a j} }}{{( - f_1 )^j }}}  = \left( { - \frac{f_2}{f_1}} \right)^k X_m\,,
\end{equation}
for $k$ a non-negative integer.

\end{lemma}

\begin{thm} 
The following identities hold for any integer $m$ and any non-negative integer~$k$:
\begin{equation}\label{eq.c0qtir4}
\sum_{j = 0}^k {( - 1)^j \binom kj2^j T_{m - 4k + 3j} }  = ( - 1)^k T_m\,,
\end{equation}
\begin{equation}\label{eq.er5gj6p}
\sum_{j = 0}^k {\binom kjT_{m - 3k + 4j} }  = 2^k T_m
\end{equation}
and
\begin{equation}\label{eq.f2qpycg}
\sum_{j = 0}^k {(-1)^j\binom kj2^{-j}T_{m+3k+j}}   = 2^{-k} T_m\,. 
\end{equation}

\end{thm}

\begin{proof}
Identify $X=T$ in Lemma~\ref{lem.i84yg3s} and use the $f_1$, $f_2$, $a$ and $b$ values found in the proof of Theorem~\ref{thm.mt8vm2p}.
\end{proof}
Particular cases of ~\eqref{eq.c0qtir4},~\eqref{eq.er5gj6p} and \eqref{eq.f2qpycg} are the following identities:
\begin{equation}
\sum_{j = 0}^k {( - 1)^j \binom kj2^j T_{3j} }  = ( - 1)^k T_{4k}\,,
\end{equation}
\begin{equation}
\sum_{j = 0}^k {\binom kjT_{4j} }  = 2^k T_{3k}
\end{equation}
and
\begin{equation}
\sum_{j = 0}^k {( - 1)^j \binom kj2^{ - j} T_j }  = 2^{ - k} (T_{3k - 1}^2  - T_{3k - 2} T_{3k} )\,.
\end{equation}
\section{Weighted double binomial sums}
\begin{lemma}\label{lem.h2de9i7}
Let $\{X_m\}$ be any arbitrary sequence, $X_m$ satisfying a third order recurrence relation $X_m=f_1X_{m-a}+f_2X_{m-b}+f_3X_{m-c}$, where $f_1$, $f_2$ and $f_3$ are arbitrary nonvanishing functions and $a$, $b$ and $c$ are integers. Then, the following identities hold:
\begin{equation}\label{eq.wgx2r2f}
\sum_{j = 0}^k {\sum_{s = 0}^j {\binom kj\binom js\left( {\frac{{f_2 }}{{f_3 }}} \right)^j\left( {\frac{{f_1 }}{{f_2 }}} \right)^s  X_{m - ck + (c - b)j + (b - a)s} } }  = \frac{{X_m }}{{f_3{}^k }}\,,
\end{equation}
\begin{equation}\label{eq.sm9bygb}
\sum_{j = 0}^k {\sum_{s = 0}^j {\binom kj\binom js\left( {\frac{{f_3 }}{{f_2 }}} \right)^j\left( {\frac{{f_1 }}{{f_3 }}} \right)^s  X_{m - bk + (b - c)j + (c - a)s} } }  = \frac{{X_m }}{{f_2{}^k }}\,,
\end{equation}
\begin{equation}
\sum_{j = 0}^k {\sum_{s = 0}^j {\binom kj\binom js\left( {\frac{{f_3 }}{{f_1 }}} \right)^j\left( {\frac{{f_2 }}{{f_3 }}} \right)^s X_{m - ak + (a - c)j + (c - b)s} } }  = \frac{{X_m }}{{f_1{}^k }}\,,
\end{equation}
\begin{equation}
\sum_{j = 0}^k {\sum_{s = 0}^j {\binom kj\binom js\left( {\frac{{f_2 }}{{f_3 }}} \right)^j\left( {-\frac{{1 }}{{f_2 }}} \right)^s X_{m - (c-a)k + (c - b)j + bs } } }  = \left(-\frac {f_1}{f_3}\right)^kX_m\,,
\end{equation}
\begin{equation}
\sum_{j = 0}^k {\sum_{s = 0}^j {\binom kj\binom js\left( {\frac{{f_1 }}{{f_3 }}} \right)^j\left( {-\frac{{1 }}{{f_1 }}} \right)^s X_{m - (c-b)k + (c - a)j + as } } }  = \left(-\frac {f_2}{f_3}\right)^kX_m\,,
\end{equation}
and
\begin{equation}\label{eq.o7540wl}
\sum_{j = 0}^k {\sum_{s = 0}^j {\binom kj\binom js\left( {\frac{{f_1 }}{{f_2 }}} \right)^j\left( {-\frac{{1 }}{{f_1 }}} \right)^s X_{m  - (b-c)k + (b - a)j + as} } }  = \left(-\frac {f_3}{f_2}\right)^kX_m\,.
\end{equation}

\end{lemma}
\begin{proof}
Only identity~\eqref{eq.wgx2r2f} needs to be proved as identities~\eqref{eq.sm9bygb}--\eqref{eq.o7540wl} are obtained from~\eqref{eq.wgx2r2f} by re-arranging the recurrence relation. The proof of~\eqref{eq.wgx2r2f} is by induction on $k$, similar to the proof of Lemma~3 of~\cite{adegoke18}.
\end{proof}
\begin{thm}\label{thm.i50wspl}
The following identities hold for non-negative integer $k$, integer $m$ and integer $r\notin \{-17,-4,-1,0\}$:
\begin{equation}\label{eq.zhg6bhz}
\sum_{j = 0}^k {\sum_{s = 0}^j {\binom kj\binom js(T_{r - 1}  + T_{r} )^{j - s} \frac{{T_{r + 1} ^s }}{{T_{r} ^j }}T_{m - (r+2)k + j + s} } }  = \frac{{T_m }}{{T_{r}^k }}\,,
\end{equation}
\begin{equation}\label{eq.w0c5uq4}
\sum_{j = 0}^k {\sum_{s = 0}^j {\binom kj\binom js\frac{{T_{r} ^{j - s} T_{r + 1} ^s }}{{(T_{r - 1}  + T_{r} )^j }}T_{m - (r + 1)k - j + 2s} } }  = \frac{{T_m }}{{(T_{r - 1}  + T_{r} )^k }}\,,
\end{equation}
\begin{equation}\label{eq.y71b7qv}
\sum_{j = 0}^k {\sum_{s = 0}^j {\binom kj\binom js\frac{{T_{r - 1} ^{j - s} (T_{r - 2}  + T_{r - 1} )^s }}{{T_{r} ^j }}T_{m - (r - 1)k - 2j + s} } }  = \frac{{T_m }}{{T_{r} ^k }}\,,
\end{equation}
\begin{equation}\label{eq.blykfib}
\sum_{j = 0}^k {\sum_{s = 0}^j {( - 1)^s \binom kj\binom js\frac{{(T_{r - 1}  + T_{r} )^{j - s} }}{{T_{r} ^j }}T_{m - 2k + j + (r + 1)s} } }  = ( - 1)^k \left( {\frac{{T_{r + 1} }}{{T_{r} }}} \right)^k T_m\,, 
\end{equation}
\begin{equation}\label{eq.bw5q3gi}
\sum_{j = 0}^k {\sum_{s = 0}^j {( - 1)^s \binom kj\binom js\frac{{T_{r + 1} ^{j - s} }}{{T_{r} ^j }}T_{m - k + 2j + rs} } }  = ( - 1)^k \left( {\frac{{T_{r - 1}  + T_{r} }}{{T_{r} }}} \right)^k T_m 
\end{equation}
and
\begin{equation}\label{eq.ajfkce1}
\sum_{j = 0}^k {\sum_{s = 0}^j {( - 1)^s \binom kj\binom js\frac{{T_{r + 1} ^{j - s} }}{{(T_{r - 1}  + T_{r} )^j }}T_{m + k + j + rs} } }  = ( - 1)^k \left( {\frac{{T_{r} }}{{T_{r - 1}  + T_{r} }}} \right)^k T_m\,.
\end{equation}

\end{thm}
\begin{proof}
Write the identity~\eqref{eq.c1sfqfe} as $T_{m}  = T_{r} T_{m-r-2}  + (T_{r - 1}  + T_{r} )T_{m - r - 1}  + T_{r+1} T_{m - r}$, identify $f_1=T_{r}$, $f_2=T_{r - 1}  + T_{r}$, $f_3=T_{r+1}$, $a=r+2$, $b=r+1$, $c=r$ and use these in Lemma~\ref{lem.h2de9i7} with $X=T$.
\end{proof}

\end{document}